\documentclass[11pt]{amsart}

\usepackage{algorithm,algpseudocodex,amssymb,booktabs,caption,chngcntr,enumitem,mathtools,nicematrix,rotating,tikz}
\usepackage[colorlinks=true,allcolors=blue]{hyperref}
\usepackage[margin=1.5in]{geometry}

\makeatletter
\g@addto@macro\bfseries{\boldmath}
\makeatother

\makeatletter
\newcommand*\vspacebeforeline[1]{%
    \ifvmode 
        \vskip #1
        \vskip \z@skip
    \else
        \@bsphack
        \vadjust pre {%
            \@restorepar
            \vskip #1
            \vskip \z@skip
        }%
        \@esphack
    \fi
}
\makeatother

\newenvironment{amssidewaysfigure}
  {\begin{figure}[ht!]\begin{turn}{90}\vspace*{.6975\textwidth}\begin{minipage}{0.95\textheight}}
  {\end{minipage}\end{turn}\end{figure}}

\newcommand{\trivialarray}
{\,\vspacebeforeline{2pt}\renewcommand{\arraystretch}{1.3}
\begin{NiceArray}{ccc}[hvlines,corners]
1 & 0 & 0
\end{NiceArray}\vspace{3pt}\,}

\DeclareMathOperator{\Sol}{Sol}

\DeclarePairedDelimiter{\floor}{\lfloor}{\rfloor}

\newtheorem{theorem}{Theorem}
\newtheorem{lemma}[theorem]{Lemma}
\newtheorem{proposition}[theorem]{Proposition}

\newcommand{\thistheoremname}{}
\newtheorem*{genericthm*}{\thistheoremname}
\newenvironment{namedthm*}[1]
  {\renewcommand{\thistheoremname}{#1}%
   \begin{genericthm*}}
  {\end{genericthm*}}

\theoremstyle{definition}
\newtheorem{move}{Move}

\theoremstyle{remark}
\newtheorem{example}[theorem]{Example}
\newtheorem{remark}[theorem]{Remark}

\makeatletter
\let\c@algorithm\c@theorem
\makeatother

\begin{document}

\title[On solutions of $\sum_{i=1}^n 1/x_i = 1$ in integers of the form $2^a k^b$]{On solutions of $\sum_{i=1}^n 1/x_i = 1$ in integers\\ of the form $2^a k^b$}

\author{Joel Louwsma}
\address{Department of Mathematics\\ Niagara University\\ Niagara University, NY 14109\\ USA}
\email{jlouwsma@niagara.edu}

\dedicatory{Dedicated to the memory of Ruth Favro (1938--2023)}

\begin{abstract}
We give an algorithm that produces all solutions of the equation $\sum_{i=1}^n 1/x_i = 1$ in integers of the form $2^a k^b$, where $k$ is a fixed positive integer that is not a power of~$2$, $a$ is an element of $\{0,1,2\}$ that can vary from term to term, and $b$ is a nonnegative integer that can vary from term to term. We also completely characterize the pairs $(k,n)$ for which this equation has a nontrivial solution in integers of this form.
\end{abstract}

\maketitle

\section{Introduction}

This paper studies solutions of the Diophantine equation
\begin{equation}\label{eq:main}
\sum_{i=1}^n\frac{1}{x_i}=1, \qquad x_1\leq\dotsb\leq x_n,
\end{equation}
in integers of the form $2^ak^b$, where $k$ is a fixed positive integer that is not a power of~$2$ and where $a$ and~$b$ are nonnegative integers that can vary from term to term. Solutions of~\eqref{eq:main} have previously been studied from a number of perspectives. For $n\leq8$, the number of solutions in positive integers is given in \cite[A002966]{OEIS} and the number of solutions in distinct positive integers is given in \cite[A006585]{OEIS}. It is a classical result of Curtiss~\cite{Curtiss} that for any solution of~\eqref{eq:main} the~$x_i$ are bounded by the $(n-1)$th term of the sequence \cite[A007018]{OEIS}. However, this sequence grows doubly exponentially, making it impractical to use a brute force approach to find all solutions for large~$n$. For general~$n$, bounds on the number of solutions are given in \cite{BE,EG,K14,S03}. 

Several authors have considered~\eqref{eq:main} with various restrictions. When the~$x_i$ are required to be distinct odd integers, Burshtein~\cite{B07} and Shiu~\cite{Shiu} showed that there are $5$~solutions when $n=9$, Arce-Nazario--Castro--Figueroa~\cite{ACF13} showed that there are 379,118 solutions when $n=11$, and Elsholtz~\cite{E16} gave bounds on the number of solutions for general~$n$. When the~$x_i$ are distinct integers of the form $3^{\alpha}5^{\beta}7^{\gamma}$, Burshtein~\cite{B08} found all solutions when $n=11$ and Arce-Nazario--Castro--Figueroa~\cite{ACF17} found all solutions when $n=13$ and $n=15$. Solutions in distinct integers of the form $pq$ and in distinct integers of the form $p^{\alpha}q^{\beta}$, where in both cases $p$ and~$q$ are prime numbers that are allowed to vary from term to term, were exhibited by Burshtein~\cite{B06}.

For a fixed positive integer~$k$, a number of authors have studied the asymptotic growth rate of the number of solutions of~\eqref{eq:main} in integers of the form~$k^a$; see~\cite{EHP} and the references therein. Given fixed prime numbers $p_1,\dotsc,p_t$, Chen--Elsholtz--Jiang~\cite{CEJ} gave bounds on the number of solutions in distinct integers of the form $p_1^{\alpha_1}\dotsm p_t^{\alpha_t}$. Levaillant~\cite{Lev} recently gave an algorithm for producing all solutions in distinct integers of the form $2^aq^b$, where $q$ is a fixed odd prime number, $a$ is an element of $\{0,1,2\}$ that can vary from term to term, and $b$ is a nonnegative integer that can vary from term to term.

In contrast to most previous work, we consider solutions of~\eqref{eq:main} where the~$x_i$ are not necessarily distinct, motivated in part by the fact that every such solution gives rise to an arithmetical structure on the complete graph~$K_n$ (see, for example, \cite[pp.~145--146]{GK}). We generalize \cite{Lev} by removing both the restriction that the~$x_i$ are distinct and the requirement that $q$ is prime. Given a positive integer~$k$ that is not a power of~$2$, let 
\[
S(k)=\{2^ak^b\mid a\in\{0,1,2\}, b\in\mathbb{Z}_{\geq0}\}.
\]
For fixed such~$k$, Algorithm~\ref{alg:main} produces all solutions of~\eqref{eq:main} with each $x_i\in S(k)$.

Unlike the approach of~\cite{Lev}, we encode a solution of this form as a collection of nonnegative integers $c_{a,b}$ that record the number of times the term $1/2^ak^b$ occurs in the solution. Since $k$ is not a power of~$2$, each element of $S(k)$ can be uniquely represented in the form $2^ak^b$, making $c_{a,b}$ well defined. We display the $c_{a,b}$ in an array where the columns are indexed from left to right by~$a$, beginning with index~$0$, and the rows are indexed from bottom to top by~$b$, beginning with index~$0$. Since for each~$x_i$ the exponent of~$2$ is in $\{0,1,2\}$, our arrays have three columns. When displaying these arrays, we always include the row indexed by~$0$, and it is assumed that all entries above those shown are~$0$. As an example, the array associated to the solution
\[
\frac{1}{2^2}+\frac{1}{2^2}+\frac{1}{2\cdot3}+\frac{1}{3^2}+\frac{1}{3^2}+\frac{1}{2^2\cdot3}+\frac{1}{2^2\cdot3^2}=1
\]
is
\[
\renewcommand{\arraystretch}{1.3}
\begin{NiceArray}{ccc}[hvlines,corners]
2 & 0 & 1\\
0 & 1 & 1\\
0 & 0 & 2
\end{NiceArray}\,.
\]

Since~\eqref{eq:main} requires that $x_1\leq\dotsb\leq x_n$ and $k$ is not a power of~$2$, every solution with each $x_i\in S(k)$ can be recovered from its associated array. Therefore it suffices to give an algorithm that produces all arrays corresponding to solutions of~\eqref{eq:main} of the desired form. Section~\ref{sec:moves} introduces a collection of moves and shows that they can be used to reduce each such array other than $\trivialarray$. Section~\ref{sec:algorithm} then shows how to generate each of these arrays from $\trivialarray$ in a unique way. Our algorithm does not easily generalize to solutions in integers of the form $2^a k^b$ where $a$ is allowed to be larger than~$2$ or to solutions in integers of the form $\ell^a k^b$ for $\ell\neq2$; this is discussed further in Remark~\ref{rem:generalizinghard}.

We say a solution of~\eqref{eq:main} is \emph{nontrivial} if $k$ divides some~$x_i$, or equivalently if the corresponding array has a nonzero entry that is not in the bottom row. A \emph{trivial} solution is one that corresponds to an array that has nonzero entries only in row~$0$. We are mainly interested in nontrivial solutions, and in Section~\ref{sec:characterize} we characterize the pairs $(k,n)$ for which there is a nontrivial solution in integers of the desired form.
\begin{namedthm*}{Theorem~\ref{thm:characterization}}
Let $k$ be a positive integer that is not a power of~$2$. Equation~\eqref{eq:main} has a nontrivial solution in integers in $S(k)$ if and only if $(k,n)=(3,3)$ or
\[
n\geq\begin{cases}
(k+8)/4 & \text{for } k\equiv0\pmod{4},\\
(k+11)/4 & \text{for } k\equiv1\pmod{4},\\
(k+10)/4 & \text{for } k\equiv2\pmod{4},\\
(k+13)/4 & \text{for } k\equiv3\pmod{4}.
\end{cases}
\]
\end{namedthm*}
The proof of Theorem~\ref{thm:characterization} constructs an explicit nontrivial solution for each pair $(k,n)$ that satisfies these conditions.

Although our approach does not require that the~$x_i$ are distinct, by checking which solutions produced by Algorithm~\ref{alg:main} have distinct $x_i$'s we also obtain an algorithm that produces all solutions of~\eqref{eq:main} in distinct integers of the same form. In the case when $k$ is a prime number, this recovers counts from~\cite{Lev}.

\section{The moves}\label{sec:moves}

In this section, we define five moves that can be used to transform arrays corresponding to solutions of~\eqref{eq:main} in integers of the form $2^ak^b$ into arrays corresponding to other such solutions, where $k$ is a fixed positive integer that is not a power of~$2$. Each of these moves can be applied in two directions; when they are applied in the forward direction we call them \emph{reduction moves}, and when they are applied in the reverse direction we call them \emph{expansion moves}. We will show that one of these reduction moves can be applied to each array other than $\trivialarray$ that corresponds to a solution of~\eqref{eq:main} in integers in $S(k)$.

\begin{move}\label{move:1}
This move adjusts two adjacent cells of an array by the rule
\[
\renewcommand{\arraystretch}{1.3}
\begin{NiceArray}{cc}[hvlines,corners]
\bullet & \bullet+2
\end{NiceArray}
\quad\longleftrightarrow\quad
\begin{NiceArray}{cc}[hvlines,corners]
\bullet+1 & \bullet
\end{NiceArray}\,,
\]
where the dots indicate nonnegative values that are unchanged. Move~\ref{move:1} encodes the relations
\[
\frac{1}{2k^b}+\frac{1}{2k^b}=\frac{1}{k^b}\quad\text{ and }\quad\frac{1}{4k^b}+\frac{1}{4k^b}=\frac{1}{2k^b}
\]
and changes the number of terms of a solution by~$1$.
\end{move}

\begin{move}\label{move:2}
The definition of this move depends on the residue of $k$ modulo~$4$. When $k\equiv0\pmod{4}$, it adjusts two consecutive rows of an array by the rule 
\[
\renewcommand{\arraystretch}{1.3}
\begin{NiceArray}{ccc}[hvlines,corners]
\bullet+\frac{k}{4} & \bullet & \bullet\\
\bullet & \bullet & \bullet
\end{NiceArray}
\quad\longleftrightarrow\quad
\begin{NiceArray}{ccc}[hvlines,corners]
\bullet & \bullet & \bullet\\
\bullet & \bullet & \bullet+1
\end{NiceArray}\,.
\]
This encodes the relation
\[
\underbrace{\frac{1}{k^{b+1}}+\dotsb+\frac{1}{k^{b+1}}}_{\text{$k/4$ times}}=\frac{1}{4k^b}
\]
and changes the number of terms of a solution by $k/4-1$. 

When $k\equiv1\pmod{4}$ or $k\equiv3\pmod{4}$, this move adjusts two consecutive rows of an array by the rule 
\[
\renewcommand{\arraystretch}{1.3}
\begin{NiceArray}{ccc}[hvlines,corners]
\bullet+k & \bullet & \bullet\\
\bullet & \bullet & \bullet
\end{NiceArray}
\quad\longleftrightarrow\quad
\begin{NiceArray}{ccc}[hvlines,corners]
\bullet & \bullet & \bullet\\
\bullet+1 & \bullet & \bullet
\end{NiceArray}\,.
\]
This encodes the relation
\[
\underbrace{\frac{1}{k^{b+1}}+\dotsb+\frac{1}{k^{b+1}}}_{\text{$k$ times}}=\frac{1}{k^b}
\]
and changes the number of terms of a solution by $k-1$. 

When $k\equiv2\pmod{4}$, this move adjusts two consecutive rows of an array by the rule 
\[
\renewcommand{\arraystretch}{1.3}
\begin{NiceArray}{ccc}[hvlines,corners]
\bullet+\frac{k}{2} & \bullet & \bullet\\
\bullet & \bullet & \bullet
\end{NiceArray}
\quad\longleftrightarrow\quad
\begin{NiceArray}{ccc}[hvlines,corners]
\bullet & \bullet & \bullet\\
\bullet & \bullet+1 & \bullet
\end{NiceArray}\,.
\]
This encodes the relation
\[
\underbrace{\frac{1}{k^{b+1}}+\dotsb+\frac{1}{k^{b+1}}}_{\text{$k/2$ times}}=\frac{1}{2k^b}
\]
and changes the number of terms of a solution by $k/2-1$. 
\end{move}

\begin{move}\label{move:3}
This move is only defined when $k$ is odd. When $k\equiv1\pmod{4}$, it adjusts two consecutive rows of an array by the rule 
\[
\renewcommand{\arraystretch}{1.3}
\begin{NiceArray}{ccc}[hvlines,corners]
\bullet+\frac{k-1}{4} & \bullet & \bullet+1\\
\bullet & \bullet & \bullet
\end{NiceArray}
\quad\longleftrightarrow\quad
\begin{NiceArray}{ccc}[hvlines,corners]
\bullet & \bullet & \bullet\\
\bullet & \bullet & \bullet+1
\end{NiceArray}\,.
\]
This encodes the relation
\[
\underbrace{\frac{1}{k^{b+1}}+\dotsb+\frac{1}{k^{b+1}}}_{\text{$(k-1)/4$ times}}+\frac{1}{4k^{b+1}}=\frac{1}{4k^{b}}
\]
and changes the number of terms of a solution by $(k-1)/4$.

When $k\equiv3\pmod{4}$, this move adjusts two consecutive rows of an array by the rule 
\[
\renewcommand{\arraystretch}{1.3}
\begin{NiceArray}{ccc}[hvlines,corners]
\bullet+\frac{3k-1}{4} & \bullet & \bullet+1\\
\bullet & \bullet & \bullet
\end{NiceArray}
\quad\longleftrightarrow\quad
\begin{NiceArray}{ccc}[hvlines,corners]
\bullet & \bullet & \bullet\\
\bullet & \bullet+1 & \bullet+1
\end{NiceArray}\,.
\]
This encodes the relation
\[
\underbrace{\frac{1}{k^{b+1}}+\dotsb+\frac{1}{k^{b+1}}}_{\text{$(3k-1)/4$ times}}+\frac{1}{4k^{b+1}}=\frac{1}{2k^{b}}+\frac{1}{4k^{b}}
\]
and changes the number of terms of a solution by $(3k-1)/4-1=(3k-5)/4$.
\end{move}

\begin{move}\label{move:4}
This move is only defined when $k$ is congruent to $1$, $2$, or~$3$ modulo~$4$. When $k\equiv1\pmod{4}$ or $k\equiv3\pmod{4}$, it adjusts two consecutive rows of an array by the rule 
\[
\renewcommand{\arraystretch}{1.3}
\begin{NiceArray}{ccc}[hvlines,corners]
\bullet+\frac{k-1}{2} & \bullet+1 & \bullet\\
\bullet & \bullet & \bullet
\end{NiceArray}
\quad\longleftrightarrow\quad
\begin{NiceArray}{ccc}[hvlines,corners]
\bullet & \bullet & \bullet\\
\bullet & \bullet+1 & \bullet
\end{NiceArray}\,.
\]
This encodes the relation
\[
\underbrace{\frac{1}{k^{b+1}}+\dotsb+\frac{1}{k^{b+1}}}_{\text{$(k-1)/2$ times}}+\frac{1}{2k^{b+1}}=\frac{1}{2k^{b}}
\]
and changes the number of terms of a solution by $(k-1)/2$.

When $k\equiv2\pmod{4}$, this move adjusts two consecutive rows of an array by the rule 
\[
\renewcommand{\arraystretch}{1.3}
\begin{NiceArray}{ccc}[hvlines,corners]
\bullet+\frac{k-2}{4} & \bullet+1 & \bullet\\
\bullet & \bullet & \bullet
\end{NiceArray}
\quad\longleftrightarrow\quad
\begin{NiceArray}{ccc}[hvlines,corners]
\bullet & \bullet & \bullet\\
\bullet & \bullet & \bullet+1
\end{NiceArray}\,.
\]
This encodes the relation
\[
\underbrace{\frac{1}{k^{b+1}}+\dotsb+\frac{1}{k^{b+1}}}_{\text{$(k-2)/4$ times}}+\frac{1}{2k^{b+1}}=\frac{1}{4k^{b}}
\]
and changes the number of terms of a solution by $(k-2)/4$.
\end{move}

\begin{move}\label{move:5}
This move is only defined when $k$ is odd. When $k\equiv1\pmod{4}$, it adjusts two consecutive rows of an array by the rule 
\[
\renewcommand{\arraystretch}{1.3}
\begin{NiceArray}{ccc}[hvlines,corners]
\bullet+\frac{3k-3}{4} & \bullet+1 & \bullet+1\\
\bullet & \bullet & \bullet
\end{NiceArray}
\quad\longleftrightarrow\quad
\begin{NiceArray}{ccc}[hvlines,corners]
\bullet & \bullet & \bullet\\
\bullet & \bullet+1 & \bullet+1
\end{NiceArray}\,.
\]
This encodes the relation
\[
\underbrace{\frac{1}{k^{b+1}}+\dotsb+\frac{1}{k^{b+1}}}_{\text{$(3k-3)/4$ times}}+\frac{1}{2k^{b+1}}+\frac{1}{4k^{b+1}}=\frac{1}{2k^{b}}+\frac{1}{4k^{b}}
\]
and changes the number of terms of a solution by $(3k-3)/4$.

When $k\equiv3\pmod{4}$, this move adjusts two consecutive rows of an array by the rule 
\[
\renewcommand{\arraystretch}{1.3}
\begin{NiceArray}{ccc}[hvlines,corners]
\bullet+\frac{k-3}{4} & \bullet+1 & \bullet+1\\
\bullet & \bullet & \bullet
\end{NiceArray}
\quad\longleftrightarrow\quad
\begin{NiceArray}{ccc}[hvlines,corners]
\bullet & \bullet & \bullet\\
\bullet & \bullet & \bullet+1
\end{NiceArray}\,.
\]
This encodes the relation
\[
\underbrace{\frac{1}{k^{b+1}}+\dotsb+\frac{1}{k^{b+1}}}_{\text{$(k-3)/4$ times}}+\frac{1}{2k^{b+1}}+\frac{1}{4k^{b+1}}=\frac{1}{4k^{b}}
\]
and changes the number of terms of a solution by $(k-3)/4+1=(k+1)/4$.
\end{move}

Note that Moves \ref{move:1}--\ref{move:5}, when applied in the forward direction for any~$k$ that is not a power of~$2$, always decrease the number of terms of the corresponding solution.

We next establish a useful lemma.

\begin{lemma}\label{lem:admissible}
Let $k$ be a positive integer that is not a power of~$2$, let $A$ be an array corresponding to a solution of~\eqref{eq:main} with each $x_i\in S(k)$, and let $\beta$ be the index of the top nonzero row of~$A$.
\begin{enumerate}[label=\textup{(\alph*)},ref=\textup{\alph*}]
\item If $\beta=0$, then $A$ is one of the following: 
\[
\renewcommand{\arraystretch}{1.3}
\begin{NiceArray}{ccc}[hvlines,corners]
1 & 0 & 0
\end{NiceArray} 
\qquad
\renewcommand{\arraystretch}{1.3}
\begin{NiceArray}{ccc}[hvlines,corners]
0 & 2 & 0
\end{NiceArray} 
\qquad
\renewcommand{\arraystretch}{1.3}
\begin{NiceArray}{ccc}[hvlines,corners]
0 & 1 & 2
\end{NiceArray} 
\qquad
\renewcommand{\arraystretch}{1.3}
\begin{NiceArray}{ccc}[hvlines,corners]
0 & 0 & 4
\end{NiceArray}\,.\label{part:trivialpossibilities} 
\]
\item If $\beta\geq1$, then $k$ divides $4c_{0,\beta}+2c_{1,\beta}+c_{2,\beta}$. \label{part:nontrivialadmissible}
\end{enumerate}
\end{lemma}

\begin{proof}
When $\beta=0$, multiplying~\eqref{eq:main} through by~$4$ gives $4c_{0,0}+2c_{1,0}+c_{2,0}=4$. Thus the only possibilities for $c_{0,0}$ are $1$ and~$0$. If $c_{0,0}=1$, then $c_{1,0}=c_{2,0}=0$. If $c_{0,0}=0$, the only possibilities for $c_{1,0}$ are $2$, $1$, and~$0$. If $c_{0,0}=0$ and $c_{1,0}=2$, then $c_{2,0}=0$. If $c_{0,0}=0$ and $c_{1,0}=1$, then $c_{2,0}=2$. If $c_{0,0}=c_{1,0}=0$, then $c_{2,0}=4$. This proves~(\ref{part:trivialpossibilities}).

When $\beta\geq1$, multiplying~\eqref{eq:main} through by $4k^\beta$ gives $4c_{0,\beta}+2c_{1,\beta}+c_{2,\beta}+sk=4k^\beta$ for some integer~$s$. We have that $k$ divides the right side of this equation, so therefore it also divides $4c_{0,\beta}+2c_{1,\beta}+c_{2,\beta}$, proving~(\ref{part:nontrivialadmissible}).
\end{proof}

We use this lemma to show that arrays corresponding to solutions of the form considered in this paper can be reduced by the above moves.

\begin{proposition}\label{prop:reductionmove}
Let $k$ be a positive integer that is not a power of~$2$, and suppose $A$ is an array corresponding to a solution of~\eqref{eq:main} with each $x_i\in S(k)$. If $A$ is not $\trivialarray$, then it can be reduced using one of Moves \ref{move:1}--\ref{move:5}. Moreover, this reduction can be chosen so that it edits the top nonzero row of~$A$.
\end{proposition}

\begin{proof}
Let $\beta$ be the index of the top nonzero row of~$A$. If $\beta=0$, then Lemma~\ref{lem:admissible}(\ref{part:trivialpossibilities}) gives the possible arrays. All except $\trivialarray$ can be reduced using Move~\ref{move:1}. 

Now suppose $\beta\geq1$. If $c_{1,\beta}\geq2$ or $c_{2,\beta}\geq2$, then $A$ can be reduced using Move~\ref{move:1}, so therefore it suffices to consider the cases when $c_{1,\beta}, c_{2,\beta}\in\{0,1\}$. Lemma~\ref{lem:admissible}(\ref{part:nontrivialadmissible}) implies that $4c_{0,\beta}+2c_{1,\beta}+c_{2,\beta}=mk$ for some positive integer~$m$, meaning $c_{0,\beta}=(mk-2c_{1,\beta}-c_{2,\beta})/4$. For each pair $(c_{1,\beta},c_{2,\beta})$, it is enough to show that a reduction move that edits row~$\beta$ can be applied for the smallest possible value of $c_{0,\beta}$; we thus consider the smallest value of~$m$ for which $(mk-2c_{1,\beta}-c_{2,\beta})/4$ is an integer. When $c_{1,\beta}=c_{2,\beta}=0$, the smallest such~$m$ is~$1$ when $k\equiv0\pmod{4}$, the smallest such~$m$ is~$4$ when $k\equiv1\pmod{4}$ or $k\equiv3\pmod{4}$, and the smallest such~$m$ is~$2$ when $k\equiv2\pmod{4}$. In each of these cases, $A$ can be reduced using Move~\ref{move:2}. When $c_{1,\beta}=0$ and $c_{2,\beta}=1$, Lemma~\ref{lem:admissible}(\ref{part:nontrivialadmissible}) only allows odd~$k$. When $k\equiv1\pmod{4}$ the smallest~$m$ is~$1$, and when $k\equiv3\pmod{4}$ the smallest~$m$ is~$3$. In both of these cases, $A$ can be reduced using Move~\ref{move:3}. When $c_{1,\beta}=1$ and $c_{2,\beta}=0$, Lemma~\ref{lem:admissible}(\ref{part:nontrivialadmissible}) excludes $k\equiv0\pmod{4}$. When $k\equiv1\pmod{4}$ or $k\equiv3\pmod{4}$ the smallest~$m$ is~$2$, and when $k\equiv2\pmod{4}$ the smallest~$m$ is~$1$. In each of these cases, $A$ can be reduced using Move~\ref{move:4}. When $c_{1,\beta}=c_{2,\beta}=1$, Lemma~\ref{lem:admissible}(\ref{part:nontrivialadmissible}) only allows odd~$k$. When $k\equiv1\pmod{4}$ the smallest~$m$ is~$3$, and when $k\equiv3\pmod{4}$ the smallest~$m$ is~$1$. In both of these cases, $A$ can be reduced using Move~\ref{move:5}. Each of the moves used here edits row~$\beta$, so this completes the proof. 
\end{proof}

\begin{remark}\label{rem:generalizinghard}
The problem considered in this paper could be generalized by fixing a positive integer~$\alpha$ and seeking solutions of~\eqref{eq:main} in integers $x_i=2^{a_i}k^{b_i}$ where the $a_i$ and~$b_i$ are nonnegative integers with each $a_i\leq\alpha$. There is a family of $2^{\alpha}+1$ reduction moves of the type given in this section such that one can be applied to each array other than $\trivialarray$ that corresponds to a solution of this form. However, these reduction moves are not guaranteed to decrease the number of terms of a solution, whereas this is essential to the algorithm given in Section~\ref{sec:algorithm}. As an example, when $k=3$ and $\alpha=3$ we have the move
\[
\renewcommand{\arraystretch}{1.3}
\begin{NiceArray}{cccc}[hvlines,corners]
\bullet+1 & \bullet & \bullet & \bullet+1\\
\bullet & \bullet & \bullet & \bullet
\end{NiceArray}
\quad\longleftrightarrow\quad
\begin{NiceArray}{cccc}[hvlines,corners]
\bullet & \bullet & \bullet & \bullet\\
\bullet & \bullet & \bullet+1 & \bullet+1
\end{NiceArray}\,,
\]
which preserves the number of terms. This move could be adjusted to
\[
\renewcommand{\arraystretch}{1.3}
\begin{NiceArray}{cccc}[hvlines,corners]
\bullet+1 & \bullet & \bullet & \bullet+1\\
\bullet & \bullet & \bullet & \bullet\\
\bullet & \bullet & \bullet & \bullet
\end{NiceArray}
\quad\longleftrightarrow\quad
\begin{NiceArray}{cccc}[hvlines,corners]
\bullet & \bullet & \bullet & \bullet\\
\bullet & \bullet & \bullet & \bullet\\
\bullet & \bullet & \bullet & \bullet+1
\end{NiceArray}\,
\]
to obtain an algorithm for $\alpha=3$, but the situation gets worse for larger~$\alpha$. When $k=3$ and $\alpha=5$, the row 
\,$\vspacebeforeline{2pt}\renewcommand{\arraystretch}{1.3}
\begin{NiceArray}{cccccc}[hvlines,corners]
1 & 0 & 0 & 0 & 0 & 1
\end{NiceArray}\vspace{3pt}$\,
can be the top nonzero row of the array corresponding to a solution, but the applicable reduction move cannot be adjusted so as to decrease the number of terms of the solution. 

The problem considered in this paper could also be generalized by fixing positive integers $\ell$ and~$k$ and considering solutions of~\eqref{eq:main} in integers $x_i=\ell^{a_i}k^{b_i}$ with each $a_i\in\{0,1,2\}$ and each~$b_i$ a nonnegative integer. There is then a family of $\ell^2+1$ reduction moves of the type given in this section such that one can be applied to each array other than $\trivialarray$ that corresponds to a solution of this form. However, these moves may not decrease the number of terms of a solution. For example, if $\ell=3$ and $k=5$, we have the move
\[
\renewcommand{\arraystretch}{1.3}
\begin{NiceArray}{ccc}[hvlines,corners]
\bullet+1 & \bullet & \bullet+1\\
\bullet & \bullet & \bullet
\end{NiceArray}
\quad\longleftrightarrow\quad
\begin{NiceArray}{ccc}[hvlines,corners]
\bullet & \bullet & \bullet\\
\bullet & \bullet & \bullet+2
\end{NiceArray}\,,
\]
which preserves the number of terms.
\end{remark}

\section{The algorithm}\label{sec:algorithm}

In this section, we give an algorithm that produces all solutions of~\eqref{eq:main} with each $x_i\in S(k)$. Our approach is to produce each array corresponding to such a solution as an expansion of $\trivialarray$ in a preferred way.

We first define a prioritization of the reduction moves given in Section~\ref{sec:moves}. Let $\beta$ be the index of the top nonzero row of the array corresponding to a solution. The prioritization of reductions that edit row~$\beta$ is as follows: 
\begin{enumerate}[label=\textup{(\arabic*)},ref=\textup{\arabic*}]
\item If $c_{2,\beta}\geq2$, reduce this entry using Move~\ref{move:1}.
\item If $c_{1,\beta}\geq2$, reduce this entry using Move~\ref{move:1}.
\item Reduce row~$\beta$ using Move~\ref{move:2} if possible.
\item Reduce row~$\beta$ using Move~\ref{move:3} if applicable and possible.
\item Reduce row~$\beta$ using Move~\ref{move:4} if applicable and possible.
\item Reduce row~$\beta$ using Move~\ref{move:5} if applicable and possible.
\end{enumerate}
Proposition~\ref{prop:reductionmove} guarantees that one of these moves can be applied to every array other than $\trivialarray$ that corresponds to a solution of~\eqref{eq:main} with each $x_i\in S(k)$. For an array corresponding to such a solution, its \emph{priority reduction} is given by reducing under the highest priority of these moves that can be applied. An expansion is called a \emph{priority expansion} if its inverse is a priority reduction.

\begin{proposition}\label{prop:priority}
Let $k$ be a positive integer that is not a power of~$2$, and suppose $A$ is an array corresponding to a solution of~\eqref{eq:main} with each $x_i\in S(k)$.
\begin{enumerate}[label=\textup{(\alph*)},ref=\textup{\alph*}]
\item There is a unique sequence of priority reductions that reduces $A$ to \linebreak$\trivialarray$. \label{part:priorityreduction}
\item There is a unique sequence of priority expansions that expands $\trivialarray$ to~$A$. \label{part:priorityexpansion}
\end{enumerate}
\end{proposition}

\begin{proof}
To prove~(\ref{part:priorityreduction}), we successively apply priority reductions, which Proposition~\ref{prop:reductionmove} guarantees can be done as long as we do not have $\trivialarray$. Each of these reductions decreases the number of terms of the corresponding solution, and this ensures that after applying finitely many priority reductions we reach $\trivialarray$. Part~(\ref{part:priorityexpansion}) follows as the sequence of priority expansions is characterized by being the inverse of the sequence of priority reductions from~(\ref{part:priorityreduction}).
\end{proof}

\begin{example}
Let $k=3$. The following is a sequence of priority reductions, where the numbers above the arrows indicate the moves used.
\[
\renewcommand{\arraystretch}{1.3}
\begin{NiceArray}{ccc}[hvlines,corners]
1 & 0 & 2\\
0 & 1 & 0
\end{NiceArray}
\quad\overset{1}{\longrightarrow}\quad
\begin{NiceArray}{ccc}[hvlines,corners]
1 & 1 & 0\\
0 & 1 & 0
\end{NiceArray}
\quad\overset{4}{\longrightarrow}\quad
\begin{NiceArray}{ccc}[hvlines,corners]
0 & 2 & 0
\end{NiceArray}
\quad\overset{1}{\longrightarrow}\quad
\begin{NiceArray}{ccc}[hvlines,corners]
1 & 0 & 0
\end{NiceArray}\,.
\]
In the reverse direction, this is a sequence of priority expansions.
\end{example}

Since Proposition~\ref{prop:priority}(\ref{part:priorityexpansion}) shows that every array corresponding to a solution of~\eqref{eq:main} with each $x_i\in S(k)$ can be obtained from $\trivialarray$ under a unique sequence of priority expansions, we obtain an procedure, described in Algorithm~\ref{alg:main}, that recursively constructs all such arrays. As solutions of~\eqref{eq:main} can be recovered from these arrays, this algorithm produces all solutions with each $x_i\in S(k)$.

\begin{algorithm}
\caption{Produce all solutions of~\eqref{eq:main} with each $x_i\in S(k)$}\label{alg:main}
\begin{algorithmic}
\Require A positive integer~$k$ that is not a power of~$2$ and a positive integer~$n$
\Ensure All solutions of~\eqref{eq:main} with each $x_i\in S(k)$
\State $\Sol(j)\coloneqq\emptyset$ for all $j\in\mathbb{Z}$
\State $\Sol(1)\coloneqq\left\{\trivialarray\right\}$
\For{$j=2,\dotsc,n$}
\ForAll{$\text{solution}\in \Sol(j-1)$}
\State $\Sol(j)\coloneqq \Sol(j)\cup\text{PriorityExpansionsUnderMove1}(\text{solution})$
\EndFor
\If{$k\equiv0\pmod{4}$}
\ForAll{$\text{solution}\in \Sol(j-(k/4-1))$}
\State $\Sol(j)\coloneqq \Sol(j)\cup\text{PriorityExpansionsUnderMove2}(\text{solution})$
\EndFor
\EndIf
\If{$k\equiv1\pmod{4}$}
\ForAll{$\text{solution}\in \Sol(j-(k-1))$}
\State $\Sol(j)\coloneqq \Sol(j)\cup\text{PriorityExpansionsUnderMove2}(\text{solution})$
\EndFor
\ForAll{$\text{solution}\in \Sol(j-(k-1)/4)$}
\State $\Sol(j)\coloneqq \Sol(j)\cup\text{PriorityExpansionsUnderMove3}(\text{solution})$
\EndFor
\ForAll{$\text{solution}\in \Sol(j-(k-1)/2)$}
\State $\Sol(j)\coloneqq \Sol(j)\cup\text{PriorityExpansionsUnderMove4}(\text{solution})$
\EndFor
\EndIf
\If{$k\equiv2\pmod{4}$}
\ForAll{$\text{solution}\in \Sol(j-(k/2-1))$}
\State $\Sol(j)\coloneqq \Sol(j)\cup\text{PriorityExpansionsUnderMove2}(\text{solution})$
\EndFor
\ForAll{$\text{solution}\in \Sol(j-(k-2)/4)$}
\State $\Sol(j)\coloneqq \Sol(j)\cup\text{PriorityExpansionsUnderMove4}(\text{solution})$
\EndFor
\EndIf
\If{$k\equiv3\pmod{4}$}
\ForAll{$\text{solution}\in \Sol(j-(k-1))$}
\State $\Sol(j)\coloneqq \Sol(j)\cup\text{PriorityExpansionsUnderMove2}(\text{solution})$
\EndFor
\ForAll{$\text{solution}\in \Sol(j-(3k-5)/4)$}
\State $\Sol(j)\coloneqq \Sol(j)\cup\text{PriorityExpansionsUnderMove3}(\text{solution})$
\EndFor
\ForAll{$\text{solution}\in \Sol(j-(k-1)/2)$}
\State $\Sol(j)\coloneqq \Sol(j)\cup\text{PriorityExpansionsUnderMove4}(\text{solution})$
\EndFor
\ForAll{$\text{solution}\in \Sol(j-(k+1)/4)$}
\State $\Sol(j)\coloneqq \Sol(j)\cup\text{PriorityExpansionsUnderMove5}(\text{solution})$
\EndFor
\EndIf
\EndFor
\Return{$\Sol(n)$}
\end{algorithmic}
\end{algorithm}

Note that Move~\ref{move:5} is not needed when $k\equiv1\pmod{4}$ because in this case any expansion under Move~\ref{move:5} can be reduced by Move~\ref{move:3} and is thus not a priority expansion. 

Algorithm~\ref{alg:main} gives all arrays corresponding to solutions of~\eqref{eq:main} with each $x_i\in S(k)$ as a tree descending from $\trivialarray$. A piece of this tree for $k=3$ is shown in Appendix~\ref{app:tree3}, and a piece of this tree for $k=5$ is shown in Appendix~\ref{app:tree5}. In both of these trees, the four trivial solutions are shown in red and the nontrivial solutions are shown in black.

We have implemented Algorithm~\ref{alg:main} in \texttt{Python} and used it to generate solutions for a number of values of $k$ and~$n$. Table~\ref{tab:nondistinctcount} records the counts of nontrivial solutions of~\eqref{eq:main} in integers in $S(k)$ for $k\leq 19$ and $n\leq14$. This algorithm is efficient in practice; the data in Table~\ref{tab:nondistinctcount} were produced in less than $1$~minute on a standard personal computer.

\begin{table}
\begin{tabular}{ccccccccccccccc}
\toprule
& \multicolumn{14}{c}{$k$}\\
\cmidrule{2-15}
$n$ & $3$ & $5$ & $6$ & $7$ & $9$ & $10$ & $11$ & $12$ & $13$ & $14$ & $15$ & $17$ & $18$ & $19$\\
\midrule
1 & 0 & 0 & 0 & 0 & 0 & 0 & 0 & 0 & 0 & 0 & 0 & 0 & 0 & 0\\
2 & 0 & 0 & 0 & 0 & 0 & 0 & 0 & 0 & 0 & 0 & 0 & 0 & 0 & 0\\
3 & 2 & 0 & 0 & 0 & 0 & 0 & 0 & 0 & 0 & 0 & 0 & 0 & 0 & 0\\
4 & 7 & 2 & 2 & 0 & 0 & 0 & 0 & 0 & 0 & 0 & 0 & 0 & 0 & 0\\
5 & 21 & 7 & 5 & 2 & 1 & 1 & 0 & 1 & 0 & 0 & 0 & 0 & 0 & 0\\
6 & 64 & 14 & 10 & 6 & 3 & 4 & 1 & 2 & 1 & 1 & 0 & 0 & 0 & 0\\
7 & 187 & 30 & 19 & 10 & 7 & 6 & 4 & 4 & 2 & 3 & 1 & 1 & 1 & 0\\
8 & 565 & 69 & 37 & 17 & 11 & 8 & 7 & 7 & 4 & 5 & 3 & 2 & 3 & 1\\
9 & 1,698 & 159 & 75 & 37 & 19 & 15 & 9 & 9 & 9 & 7 & 5 & 3 & 4 & 3\\
10 & 5,140 & 361 & 150 & 64 & 34 & 25 & 15 & 14 & 12 & 10 & 8 & 6 & 6 & 4\\
11 & 15,561 & 822 & 301 & 127 & 62 & 42 & 26 & 22 & 16 & 14 & 11 & 10 & 9 & 6\\
12 & 47,188 & 1,886 & 605 & 263 & 111 & 75 & 43 & 38 & 29 & 21 & 15 & 13 & 11 & 10\\
13 & 143,138 & 4,322 & 1,216 & 493 & 201 & 120 & 74 & 62 & 46 & 35 & 24 & 18 & 15 & 12\\
14 & 434,378 & 9,897 & 2,448 & 967 & 368 & 203 & 124 & 98 & 73 & 58 & 37 & 27 & 21 & 16\\
\bottomrule
\end{tabular}
\caption{Number of nontrivial solutions of~\eqref{eq:main} in integers in $S(k)$ for $k\leq 19$ and $n\leq14$.\label{tab:nondistinctcount}}
\end{table}

By checking which of the solutions counted in Table~\ref{tab:nondistinctcount} satisfy the additional condition that the~$x_i$ are distinct, we also obtain counts of nontrivial solutions of~\eqref{eq:main} in distinct integers in $S(k)$. These counts for $n\leq14$ are shown in Table~\ref{tab:distinctcount}; for prime~$k$, they recover results of Levaillant~\cite{Lev}. When $k\geq9$, all expansions result in an array with an entry at least~$2$, meaning there are no nontrivial solutions in distinct integers.

\begin{table}
\begin{tabular}{cccccc}
\toprule
& \multicolumn{5}{c}{$k$}\\
\cmidrule{2-6}
$n$ & $3$ & $5$ & $6$ & $7$ & $\geq9$\\
\midrule
1 & 0 & 0 & 0 & 0 & 0\\
2 & 0 & 0 & 0 & 0 & 0\\
3 & 1 & 0 & 0 & 0 & 0\\
4 & 2 & 1 & 1 & 0 & 0\\
5 & 4 & 1 & 0 & 1 & 0\\
6 & 8 & 1 & 1 & 0 & 0\\
7 & 14 & 1 & 0 & 1 & 0\\
8 & 28 & 1 & 1 & 0 & 0\\
9 & 52 & 1 & 0 & 1 & 0\\
10 & 100 & 1 & 1 & 0 & 0\\
11 & 190 & 1 & 0 & 1 & 0\\
12 & 362 & 1 & 1 & 0 & 0\\
13 & 690 & 1 & 0 & 1 & 0\\
14 & 1,314 & 1 & 1 & 0 & 0\\
\bottomrule
\end{tabular}
\caption{Number of nontrivial solutions of~\eqref{eq:main} in distinct integers in $S(k)$ for $n\leq 14$.\label{tab:distinctcount}}
\end{table}

\section{When there exist nontrivial solutions}\label{sec:characterize}

Observe that, in Table~\ref{tab:nondistinctcount}, there are no solutions in the upper right but are always solutions in the lower left. In this section, we turn this observation into a precise characterization of which pairs $(k,n)$ correspond to nontrivial solutions.

\begin{theorem}\label{thm:characterization}
Let $k$ be a positive integer that is not a power of~$2$. Equation~\eqref{eq:main} has a nontrivial solution in integers in $S(k)$ if and only if $(k,n)=(3,3)$ or
\[
n\geq\begin{cases}
(k+8)/4 & \text{for } k\equiv0\pmod{4},\\
(k+11)/4 & \text{for } k\equiv1\pmod{4},\\
(k+10)/4 & \text{for } k\equiv2\pmod{4},\\
(k+13)/4 & \text{for } k\equiv3\pmod{4}.
\end{cases}
\]
\end{theorem}

We prove Theorem~\ref{thm:characterization} by proving Propositions \ref{prop:sufficient} and~\ref{prop:necessary}.

\begin{proposition}\label{prop:sufficient}
Let $k$ be a positive integer that is not a power of~$2$. If $(k,n)=(3,3)$ or
\[
n\geq\begin{cases}
(k+8)/4 & \text{for } k\equiv0\pmod{4},\\
(k+11)/4 & \text{for } k\equiv1\pmod{4},\\
(k+10)/4 & \text{for } k\equiv2\pmod{4},\\
(k+13)/4 & \text{for } k\equiv3\pmod{4},
\end{cases}
\]
then~\eqref{eq:main} has a nontrivial solution in integers in $S(k)$.
\end{proposition}

\begin{proof}
For each pair $(k,n)$ satisfying these conditions, we construct an array corresponding to a nontrivial solution of~\eqref{eq:main} in integers in $S(k)$. If $(k,n)=(3,3)$, we have the array 
\[
\renewcommand{\arraystretch}{1.3}\begin{NiceArray}{ccc}[hvlines,corners]
1 & 1 & 0\\
0 & 1 & 0
\end{NiceArray}\,.
\]
This array can be obtained from $\trivialarray$ via an expansion under Move~\ref{move:1} followed by an expansion under Move~\ref{move:4}, so hence it represents a solution. 

If $n\geq(k+8)/4$ and $k\equiv0\pmod{4}$, let
\[
s=\floor*{\frac{n-1}{k/4+1}}
\]
and 
\[
t=n-1-s(k/4+1).
\]
The assumption $n\geq(k+8)/4$ means that $s$ is at least~$1$. Also, $t$ is the residue of $n-1$ modulo $k/4+1$, so $0\leq t\leq k/4$. We have the array
\[
\renewcommand{\arraystretch}{1.3}
\begin{NiceArray}{ccc}[hvlines,corners]
\frac{k}{4}-t & 2t & 0\\
\frac{k}{4}-1 & 1 & 1\\
\vdots & \vdots & \vdots\\
\frac{k}{4}-1 & 1 & 1\\
0 & 1 & 1
\end{NiceArray}\,,
\]
where the number of rows is $s+1$. The number of terms is thus
\[
s(k/4+1)+t+1=n.
\]
This array can be obtained from $\trivialarray$ via $s$~sets of two expansions under Move~\ref{move:1} followed by one expansion under Move~\ref{move:2}, all followed by $t$~expansions under Move~\ref{move:1}. Therefore it represents a solution.

If $n\geq(k+11)/4$ and $k\equiv1\pmod{4}$, let
\[
s=\floor*{\frac{n-3}{(k-1)/4}}
\]
and 
\[
t=n-3-s\biggl(\frac{k-1}{4}\biggr).
\]
The assumption $n\geq(k+11)/4$ means that $s$ is at least~$1$. Also, $t$ is the residue of $n-3$ modulo $(k-1)/4$, so $0\leq t\leq (k-5)/4$. We have the array
\[
\renewcommand{\arraystretch}{1.3}
\begin{NiceArray}{ccc}[hvlines,corners]
\frac{k-1}{4}-t & 2t & 1\\
\frac{k-1}{4} & 0 & 0\\
\vdots & \vdots & \vdots\\
\frac{k-1}{4} & 0 & 0\\
0 & 1 & 1
\end{NiceArray}\,,
\]
where the number of rows is $s+1$. The number of terms is thus
\[
s\biggl(\frac{k-1}{4}\biggr)+t+3=n.
\]
This array can be obtained from $\trivialarray$ via two expansions under Move~\ref{move:1} followed by $s$~expansions under Move~\ref{move:3} followed by $t$~expansions under Move~\ref{move:1}, so therefore it represents a solution.

If $n\geq(k+10)/4$ and $k\equiv2\pmod{4}$, let
\[
s=\floor*{\frac{n-2}{(k+2)/4}}
\]
and 
\[
t=n-2-s\biggl(\frac{k+2}{4}\biggr).
\]
The assumption $n\geq(k+10)/4$ means that $s$ is at least~$1$. Also, $t$ is the residue of $n-2$ modulo $(k+2)/4$, so $0\leq t\leq (k-2)/4$. We have the array
\[
\renewcommand{\arraystretch}{1.3}
\begin{NiceArray}{ccc}[hvlines,corners]
\frac{k-2}{4}-t & 1+2t & 0\\
\frac{k-2}{4} & 0 & 1\\
\vdots & \vdots & \vdots\\
\frac{k-2}{4} & 0 & 1\\
0 & 1 & 1
\end{NiceArray}\,,
\]
where the number of rows is $s+1$. The number of terms is thus
\[
s\biggl(\frac{k+2}{4}\biggr)+t+2=n.
\]
This array can be obtained from $\trivialarray$ via one expansion under Move~\ref{move:1} followed by $s$~sets of one expansion under Move~\ref{move:1} followed by one expansion under Move~\ref{move:4}, all followed by $t$~expansions under Move~\ref{move:1}. Therefore it represents a solution.

If $n\geq(k+13)/4$ and $k\equiv3\pmod{4}$, let 
\[
s=\floor*{\frac{n-3}{(k+1)/4}}
\]
and 
\[
t=n-3-s\biggl(\frac{k+1}{4}\biggr).
\]
The assumption $n\geq(k+13)/4$ means that $s$ is at least~$1$. Also, $t$ is the residue of $n-3$ modulo $(k+1)/4$, so $0\leq t\leq (k-3)/4$. We have the array
\[
\renewcommand{\arraystretch}{1.3}
\begin{NiceArray}{ccc}[hvlines,corners]
\frac{k-3}{4}-t & 1+2t & 1\\
\frac{k-3}{4} & 1 & 0\\
\vdots & \vdots & \vdots\\
\frac{k-3}{4} & 1 & 0\\
0 & 1 & 1
\end{NiceArray}\,,
\]
where the number of rows is $s+1$. The number of terms is thus
\[
s\biggl(\frac{k+1}{4}\biggr)+t+3=n.
\] 
This array can be obtained from $\trivialarray$ via two expansions under Move~\ref{move:1} followed by $s$~expansions under Move~\ref{move:5} followed by $t$~expansions under Move~\ref{move:1}, so therefore it represents a solution.
\end{proof}

Before proving the converse of Proposition~\ref{prop:sufficient}, we establish several lemmas. 

\begin{lemma}\label{lem:numberofterms}
Let $k$ be a positive integer that is not a power of~$2$, let $A$ be an array corresponding to a solution of~\eqref{eq:main} with each $x_i\in S(k)$, and let $\beta$ be the index of the top nonzero row of~$A$.
\begin{enumerate}[label=\textup{(\alph*)},ref=\textup{\alph*}]
\item If $\beta=0$ and $A$ is of the form 
\,$\vspacebeforeline{2pt}\renewcommand{\arraystretch}{1.3}
\begin{NiceArray}{ccc}[hvlines,corners]
c_{0,0} & c_{1,0} & c_{2,0}+1
\end{NiceArray}\vspace{3pt}$\,
with each $c_{i,0}\geq0$, then $c_{0,0}+c_{1,0}+c_{2,0}\geq2$.\label{part:numberoftermstrivial1}
\item If $\beta=0$ and $A$ is of the form 
\,$\vspacebeforeline{2pt}\renewcommand{\arraystretch}{1.3}
\begin{NiceArray}{ccc}[hvlines,corners]
c_{0,0} & c_{1,0}+1 & c_{2,0}
\end{NiceArray}\vspace{3pt}$\,
with each $c_{i,0}\geq0$, then $c_{0,0}+c_{1,0}+c_{2,0}\geq1$.\label{part:numberoftermstrivial2}
\item If $\beta\geq1$ and row~$\beta$ of~$A$ is of the form 
\,$\vspacebeforeline{2pt}\renewcommand{\arraystretch}{1.3}
\begin{NiceArray}{ccc}[hvlines,corners]
c_{0,\beta} & c_{1,\beta} & c_{2,\beta}+1
\end{NiceArray}\vspace{3pt}$\,
with each $c_{i,\beta}\geq0$, then $c_{0,\beta}+c_{1,\beta}+c_{2,\beta}\geq1$.\label{part:numberoftermsnontrivial1}
\item If $\beta\geq1$ and row~$\beta$ of~$A$ is of the form 
\,$\vspacebeforeline{2pt}\renewcommand{\arraystretch}{1.3}
\begin{NiceArray}{ccc}[hvlines,corners]
c_{0,\beta} & c_{1,\beta}+1 & c_{2,\beta}
\end{NiceArray}\vspace{3pt}$\,
with each $c_{i,\beta}\geq0$, then $c_{0,\beta}+c_{1,\beta}+c_{2,\beta}\geq1$.\label{part:numberoftermsnontrivial2}
\end{enumerate}
\end{lemma}

\begin{proof}
If a solution corresponds to an array of the form
\,$\vspacebeforeline{2pt}\renewcommand{\arraystretch}{1.3}
\begin{NiceArray}{ccc}[hvlines,corners]
c_{0,0} & c_{1,0} & c_{2,0}+1
\end{NiceArray}\vspace{3pt}$\,
with each $c_{i,0}\geq0$, Lemma~\ref{lem:admissible}(\ref{part:trivialpossibilities}) gives that either $c_{0,0}=0, c_{1,0}=c_{2,0}=1$ or $c_{0,0}=c_{1,0}=0,\allowbreak c_{2,0}=3$. In both cases, $c_{0,0}+c_{1,0}+c_{2,0}\geq2$.

If a solution corresponds to an array of the form
\,$\vspacebeforeline{2pt}\renewcommand{\arraystretch}{1.3}
\begin{NiceArray}{ccc}[hvlines,corners]
c_{0,0} & c_{1,0}+1 & c_{2,0}
\end{NiceArray}\vspace{3pt}$\,
with each $c_{i,0}\geq0$, Lemma~\ref{lem:admissible}(\ref{part:trivialpossibilities}) gives that either $c_{0,0}=c_{2,0}=0, c_{1,0}=1$ or $c_{0,0}=c_{1,0}=0,\allowbreak c_{2,0}=2$. In both cases, $c_{0,0}+c_{1,0}+c_{2,0}\geq1$.

If a nontrivial solution corresponds to an array with top nonzero row of the form
\,$\vspacebeforeline{2pt}\renewcommand{\arraystretch}{1.3}
\begin{NiceArray}{ccc}[hvlines,corners]
c_{0,\beta} & c_{1,\beta} & c_{2,\beta}+1
\end{NiceArray}\vspace{3pt}$\,
with each $c_{i,\beta}\geq0$, Lemma~\ref{lem:admissible}(\ref{part:nontrivialadmissible}) gives that $k$ divides $4c_{0,\beta}+2c_{1,\beta}+c_{2,\beta}+1$. Since $k$ cannot be~$1$, this is only possible if $c_{0,\beta}+c_{1,\beta}+c_{2,\beta}\geq1$.

If a nontrivial solution corresponds to an array with top nonzero row of the form
\,$\vspacebeforeline{2pt}\renewcommand{\arraystretch}{1.3}
\begin{NiceArray}{ccc}[hvlines,corners]
c_{0,\beta} & c_{1,\beta}+1 & c_{2,\beta}
\end{NiceArray}\vspace{3pt}$\,
with each $c_{i,\beta}\geq0$, Lemma~\ref{lem:admissible}(\ref{part:nontrivialadmissible}) gives that $k$ divides $4c_{0,\beta}+2(c_{1,\beta}+1)+c_{2,\beta}$. Since $k$ cannot be $1$ or~$2$, this is only possible if $c_{0,\beta}+c_{1,\beta}+c_{2,\beta}\geq1$.
\end{proof}

The next lemma allows us to restrict attention to arrays whose top nonzero row can become zero under a single reduction move.

\enlargethispage*{\baselineskip}

\begin{lemma}\label{lem:onereduction}
Let $k$ be a positive integer that is not a power of~$2$, let $A$ be an array corresponding to a nontrivial solution of~\eqref{eq:main} in integers in $S(k)$, and let $\beta$ be the index of the top nonzero row of~$A$. There is a solution of~\eqref{eq:main} with each $x_i\in S(k)$ that has no more terms than the solution corresponding to~$A$, whose corresponding array~$A'$ also has top nonzero row indexed by~$\beta$, and where row~$\beta$ of~$A'$ can become 
\,$\vspacebeforeline{2pt}\renewcommand{\arraystretch}{1.3}
\begin{NiceArray}{ccc}[hvlines,corners]
0 & 0 & 0
\end{NiceArray}\vspace{3pt}$\,
under a single application of one of Moves \ref{move:2}--\ref{move:5}.
\end{lemma}

\begin{proof}
By Proposition~\ref{prop:priority}, we know $A$ can be reduced to $\trivialarray$ under a sequence of Moves \ref{move:1}--\ref{move:5}, each of which reduce the number of terms of the corresponding solution. Since the original solution is nontrivial, $\beta\geq1$, so at some step of this process row~$\beta$ becomes 
\,$\vspacebeforeline{2pt}\renewcommand{\arraystretch}{1.3}
\begin{NiceArray}{ccc}[hvlines,corners]
0 & 0 & 0
\end{NiceArray}\vspace{3pt}$\,.
This must happen after an application of one of Moves \ref{move:2}--\ref{move:5}. Performing reductions up to but not including this step, we obtain an array~$A'$ corresponding to a solution with no more terms than the original solution, where row~$\beta$ is the top nonzero row of~$A'$ and can become
\,$\vspacebeforeline{2pt}\renewcommand{\arraystretch}{1.3}
\begin{NiceArray}{ccc}[hvlines,corners]
0 & 0 & 0
\end{NiceArray}\vspace{3pt}$\,
under a single application of one of Moves \ref{move:2}--\ref{move:5}. 
\end{proof}

We use the preceding lemmas to show that every nontrivial solution has at least~$3$ terms.

\begin{lemma}\label{lem:nontrivial3terms}
Let $k$ be a positive integer that is not a power of~$2$. Every nontrivial solution of~\eqref{eq:main} in integers in $S(k)$ has at least~$3$ terms.
\end{lemma}

\begin{proof}
By Lemma~\ref{lem:onereduction}, it suffices to consider solutions corresponding to arrays whose top nonzero row can be reduced to 
\,$\vspacebeforeline{2pt}\renewcommand{\arraystretch}{1.3}
\begin{NiceArray}{ccc}[hvlines,corners]
0 & 0 & 0
\end{NiceArray}\vspace{3pt}$\,
under a single application of one of Moves \ref{move:2}--\ref{move:5}. Let $\beta$ be the index of the top nonzero row of such an array. The conclusion is satisfied if the entries of row~$\beta$ sum to at least~$3$, so it suffices to consider the cases where the entries of row~$\beta$ sum to at most~$2$. In these cases, we use Lemma~\ref{lem:admissible}(\ref{part:nontrivialadmissible}) to restrict the possibilities. There are no possibilities when $k\equiv0\pmod{4}$. If $k\equiv1\pmod{4}$, the only possibility is if $k=5$ and row~$\beta$ is
\,$\vspacebeforeline{2pt}\renewcommand{\arraystretch}{1.3}
\begin{NiceArray}{ccc}[hvlines,corners]
1 & 0 & 1
\end{NiceArray}\vspace{3pt}$\,.
After applying Move~\ref{move:3}, the top nonzero row is 
\,$\vspacebeforeline{2pt}\renewcommand{\arraystretch}{1.3}
\begin{NiceArray}{ccc}[hvlines,corners]
c_{0,\beta-1} & c_{1,\beta-1} & c_{2,\beta-1}+1
\end{NiceArray}\vspace{3pt}$\,.
By Lemma \ref{lem:numberofterms}(\ref{part:numberoftermstrivial1},\ref{part:numberoftermsnontrivial1}), we have $c_{0,\beta-1}+c_{1,\beta-1}+c_{2,\beta-1}\geq1$, meaning the original solution has at least~$3$ terms. If $k\equiv2\pmod{4}$, the only possibility is if $k=6$ and row~$\beta$ is
\,$\vspacebeforeline{2pt}\renewcommand{\arraystretch}{1.3}
\begin{NiceArray}{ccc}[hvlines,corners]
1 & 1 & 0
\end{NiceArray}\vspace{3pt}$\,.
After applying Move~\ref{move:4}, the top nonzero row is 
\,$\vspacebeforeline{2pt}\renewcommand{\arraystretch}{1.3}
\begin{NiceArray}{ccc}[hvlines,corners]
c_{0,\beta-1} & c_{1,\beta-1} & c_{2,\beta-1}+1
\end{NiceArray}\vspace{3pt}$\,.
By Lemma \ref{lem:numberofterms}(\ref{part:numberoftermstrivial1},\ref{part:numberoftermsnontrivial1}), we have $c_{0,\beta-1}+c_{1,\beta-1}+c_{2,\beta-1}\geq1$, meaning the original solution has at least~$3$ terms. If $k\equiv3\pmod{4}$, the only possibilities are if $k=3$ and row~$\beta$ is either
\,$\vspacebeforeline{2pt}\renewcommand{\arraystretch}{1.3}
\begin{NiceArray}{ccc}[hvlines,corners]
1 & 1 & 0
\end{NiceArray}\vspace{3pt}$\,
or
\,$\vspacebeforeline{2pt}\renewcommand{\arraystretch}{1.3}
\begin{NiceArray}{ccc}[hvlines,corners]
0 & 1 & 1
\end{NiceArray}\vspace{3pt}$\,.
In the first of these cases, after applying Move~\ref{move:4}, the top nonzero row is
\,$\vspacebeforeline{2pt}\renewcommand{\arraystretch}{1.3}
\begin{NiceArray}{ccc}[hvlines,corners]
c_{0,\beta-1} & c_{1,\beta-1}+1 & c_{2,\beta-1}
\end{NiceArray}\vspace{3pt}$\,.
By Lemma \ref{lem:numberofterms}(\ref{part:numberoftermstrivial2},\ref{part:numberoftermsnontrivial2}), we have $c_{0,\beta-1}+c_{1,\beta-1}+c_{2,\beta-1}\geq1$, meaning the original solution has at least~$3$ terms. In the second case, after applying Move~\ref{move:5}, the top nonzero row is
\,$\vspacebeforeline{2pt}\renewcommand{\arraystretch}{1.3}
\begin{NiceArray}{ccc}[hvlines,corners]
c_{0,\beta-1} & c_{1,\beta-1} & c_{2,\beta-1}+1
\end{NiceArray}\vspace{3pt}$\,.
By Lemma \ref{lem:numberofterms}(\ref{part:numberoftermstrivial1},\ref{part:numberoftermsnontrivial1}), we have $c_{0,\beta-1}+c_{1,\beta-1}+c_{2,\beta-1}\geq1$, meaning the original solution has at least~$3$ terms.
\end{proof}

We now prove the converse of Proposition~\ref{prop:sufficient}.

\begin{proposition}\label{prop:necessary}
Let $k$ be a positive integer that is not a power of~$2$. If~\eqref{eq:main} has a nontrivial solution in integers in $S(k)$, then $(k,n)=(3,3)$ or
\[
n\geq\begin{cases}
(k+8)/4 & \text{for } k\equiv0\pmod{4},\\
(k+11)/4 & \text{for } k\equiv1\pmod{4},\\
(k+10)/4 & \text{for } k\equiv2\pmod{4},\\
(k+13)/4 & \text{for } k\equiv3\pmod{4}.
\end{cases}
\]
\end{proposition}

\begin{proof}
If the conclusion is satisfied for a pair $(k,n)$, then it is also satisfied for the same~$k$ and any larger~$n$. Therefore, using Lemma~\ref{lem:onereduction}, it suffices to consider nontrivial solutions corresponding to arrays whose top nonzero row can be reduced to 
\,$\vspacebeforeline{2pt}\renewcommand{\arraystretch}{1.3}
\begin{NiceArray}{ccc}[hvlines,corners]
0 & 0 & 0
\end{NiceArray}\vspace{3pt}$\,
under a single reduction move. Let $A$ be an array corresponding to such a solution, and let $\beta$ be the index of the top nonzero row of~$A$. Since the solution is nontrivial, $\beta\geq1$.

In the case of 
\,$\vspacebeforeline{2pt}\renewcommand{\arraystretch}{1.3}
\begin{NiceArray}{ccc}[hvlines,corners]
\frac{k}{4} & 0 & 0
\end{NiceArray}\vspace{3pt}$\, 
and $k\equiv0\pmod{4}$, after applying Move~\ref{move:2} we obtain an array~$A'$ with top nonzero row
\,$\vspacebeforeline{2pt}\renewcommand{\arraystretch}{1.3}
\begin{NiceArray}{ccc}[hvlines,corners]
c_{0,\beta-1} & c_{1,\beta-1} & c_{2,\beta-1}+1
\end{NiceArray}\vspace{3pt}$\,.
If $\beta=1$, Lemma~\ref{lem:numberofterms}(\ref{part:numberoftermstrivial1}) gives that $c_{0,0}+c_{1,0}+c_{2,0}\geq2$. If $\beta\geq2$, Lemma~\ref{lem:nontrivial3terms} gives that the solution corresponding to~$A'$ has at least~$3$ terms, meaning the sum of the entries in row $\beta-1$ and below of~$A$ is at least~$2$. For all~$\beta$, we therefore have that $n\geq k/4+2=(k+8)/4$.

In the case of 
\,$\vspacebeforeline{2pt}\renewcommand{\arraystretch}{1.3}
\begin{NiceArray}{ccc}[hvlines,corners]
k & 0 & 0
\end{NiceArray}\vspace{3pt}$\, 
for $k\equiv1\pmod{4}$ or $k\equiv3\pmod{4}$, we have that $n\geq k$. When $k\equiv1\pmod{4}$, since $k\geq5$, this implies that $n\geq(k+11)/4$. When $k\equiv3\pmod{4}$, it implies that $(k,n)=(3,3)$ or $n\geq(k+13)/4$. 

In the case of 
\,$\vspacebeforeline{2pt}\renewcommand{\arraystretch}{1.3}
\begin{NiceArray}{ccc}[hvlines,corners]
\frac{k}{2} & 0 & 0
\end{NiceArray}\vspace{3pt}$\, 
and $k\equiv2\pmod{4}$, after applying Move~\ref{move:2} we obtain an array~$A'$ with top nonzero row
\,$\vspacebeforeline{2pt}\renewcommand{\arraystretch}{1.3}
\begin{NiceArray}{ccc}[hvlines,corners]
c_{0,\beta-1} & c_{1,\beta-1}+1 & c_{2,\beta-1}
\end{NiceArray}\vspace{3pt}$\,.
If $\beta=1$, Lemma~\ref{lem:numberofterms}(\ref{part:numberoftermstrivial2}) gives that $c_{0,0}+c_{1,0}+c_{2,0}\geq1$. If $\beta\geq2$, Lemma~\ref{lem:nontrivial3terms} gives that the solution corresponding to~$A'$ has at least~$3$ terms, meaning the sum of the entries in row $\beta-1$ and below of~$A$ is at least~$2$. For all~$\beta$, since $k\geq6$, we have that $n\geq k/2+1\geq (k+10)/4$.

In the case of 
\,$\vspacebeforeline{2pt}\renewcommand{\arraystretch}{1.3}
\begin{NiceArray}{ccc}[hvlines,corners]
\frac{k-1}{4} & 0 & 1
\end{NiceArray}\vspace{3pt}$\, 
and $k\equiv1\pmod{4}$, after applying Move~\ref{move:3} we obtain an array~$A'$ with top nonzero row
\,$\vspacebeforeline{2pt}\renewcommand{\arraystretch}{1.3}
\begin{NiceArray}{ccc}[hvlines,corners]
c_{0,\beta-1} & c_{1,\beta-1} & c_{2,\beta-1}+1
\end{NiceArray}\vspace{3pt}$\,.
If $\beta=1$, Lemma~\ref{lem:numberofterms}(\ref{part:numberoftermstrivial1}) gives that $c_{0,0}+c_{1,0}+c_{2,0}\geq2$. If $\beta\geq2$, Lemma~\ref{lem:nontrivial3terms} gives that the solution corresponding to~$A'$ has at least~$3$ terms, meaning the sum of the entries in row $\beta-1$ and below of~$A$ is at least~$2$. For all~$\beta$, we therefore have that $n\geq (k-1)/4+3=(k+11)/4$.

In the case of 
\,$\vspacebeforeline{2pt}\renewcommand{\arraystretch}{1.3}
\begin{NiceArray}{ccc}[hvlines,corners]
\frac{3k-1}{4} & 0 & 1
\end{NiceArray}\vspace{3pt}$\, 
and $k\equiv3\pmod{4}$, after applying Move~\ref{move:3} we obtain an array~$A'$ with top nonzero row
\,$\vspacebeforeline{2pt}\renewcommand{\arraystretch}{1.3}
\begin{NiceArray}{ccc}[hvlines,corners]
c_{0,\beta-1} & c_{1,\beta-1}+1 & c_{2,\beta-1}+1
\end{NiceArray}\vspace{3pt}$\,.
When $\beta=1$, Lemma~\ref{lem:admissible}(\ref{part:trivialpossibilities}) gives that $c_{0,0}+c_{1,0}+c_{2,0}\geq1$. If $\beta\geq2$, Lemma~\ref{lem:nontrivial3terms} gives that the solution corresponding to~$A'$ has at least~$3$ terms, meaning the sum of the entries in row $\beta-1$ and below of~$A$ is at least~$1$. For all~$\beta$, we therefore have that $n\geq(3k-1)/4+2\geq(k+13)/4$.

In the case of 
\,$\vspacebeforeline{2pt}\renewcommand{\arraystretch}{1.3}
\begin{NiceArray}{ccc}[hvlines,corners]
\frac{k-1}{2} & 1 & 0
\end{NiceArray}\vspace{3pt}$\, 
for $k\equiv1\pmod{4}$ or $k\equiv3\pmod{4}$, after applying Move~\ref{move:4} we obtain an array~$A'$ with top nonzero row
\,$\vspacebeforeline{2pt}\renewcommand{\arraystretch}{1.3}
\begin{NiceArray}{ccc}[hvlines,corners]
c_{0,\beta-1} & c_{1,\beta-1}+1 & c_{2,\beta-1}
\end{NiceArray}\vspace{3pt}$\,.
When $\beta=1$, Lemma~\ref{lem:admissible}(\ref{part:trivialpossibilities}) gives that either $c_{0,0}=c_{2,0}=0, c_{1,0}=1$ or $c_{0,0}=c_{1,0}=0,\allowbreak c_{2,0}=2$. If $\beta\geq2$, Lemma~\ref{lem:nontrivial3terms} gives that the solution corresponding to~$A'$ has at least~$3$ terms, meaning the sum of the entries in row $\beta-1$ and below of~$A$ is at least~$2$. For all~$\beta$, we therefore have that $(k,n)=(3,3)$ or $n\geq(k-1)/2+3=(k+5)/2$. When $k\equiv1\pmod{4}$, this implies that $n\geq(k+11)/4$. When $k\equiv3\pmod{4}$, it implies that $(k,n)=(3,3)$ or $n\geq(k+13)/4$. 

In the case of
\,$\vspacebeforeline{2pt}\renewcommand{\arraystretch}{1.3}
\begin{NiceArray}{ccc}[hvlines,corners]
\frac{k-2}{4} & 1 & 0
\end{NiceArray}\vspace{3pt}$\,
and $k\equiv2\pmod{4}$, after applying Move~\ref{move:4} we obtain an array~$A'$ with top nonzero row
\,$\vspacebeforeline{2pt}\renewcommand{\arraystretch}{1.3}
\begin{NiceArray}{ccc}[hvlines,corners]
c_{0,\beta-1} & c_{1,\beta-1} & c_{2,\beta-1}+1
\end{NiceArray}\vspace{3pt}$\,.
If $\beta=1$, Lemma~\ref{lem:numberofterms}(\ref{part:numberoftermstrivial1}) gives that $c_{0,0}+c_{1,0}+c_{2,0}\geq2$. If $\beta\geq2$, Lemma~\ref{lem:nontrivial3terms} gives that the solution corresponding to~$A'$ has at least~$3$ terms, meaning the sum of the entries in row $\beta-1$ and below of~$A$ is at least~$2$. For all~$\beta$, we therefore have that $n\geq (k-2)/4+3=(k+10)/4$.

In the case of 
\,$\vspacebeforeline{2pt}\renewcommand{\arraystretch}{1.3}
\begin{NiceArray}{ccc}[hvlines,corners]
\frac{3k-3}{4} & 1 & 1
\end{NiceArray}\vspace{3pt}$\, 
and $k\equiv1\pmod{4}$, since $k\geq5$, we have that $n\geq(3k-3)/4+2\geq(k+11)/4$. 

In the case of 
\,$\vspacebeforeline{2pt}\renewcommand{\arraystretch}{1.3}
\begin{NiceArray}{ccc}[hvlines,corners]
\frac{k-3}{4} & 1 & 1
\end{NiceArray}\vspace{3pt}$\, 
and $k\equiv3\pmod{4}$, after applying Move~\ref{move:5} we obtain an array~$A'$ with top nonzero row
\,$\vspacebeforeline{2pt}\renewcommand{\arraystretch}{1.3}
\begin{NiceArray}{ccc}[hvlines,corners]
c_{0,\beta-1} & c_{1,\beta-1} & c_{2,\beta-1}+1
\end{NiceArray}\vspace{3pt}$\,. 
If $\beta=1$, Lemma~\ref{lem:numberofterms}(\ref{part:numberoftermstrivial1}) gives that $c_{0,0}+c_{1,0}+c_{2,0}\geq2$. If $\beta\geq2$, Lemma~\ref{lem:nontrivial3terms} gives that the solution corresponding to~$A'$ has at least~$3$ terms, meaning the sum of the entries in row $\beta-1$ and below of~$A$ is at least~$2$. For all~$\beta$, we therefore have that $n\geq (k-3)/4+4=(k+13)/4$.
\end{proof}

Propositions \ref{prop:sufficient} and~\ref{prop:necessary} together show Theorem~\ref{thm:characterization}.

\section*{Code availability}

We have implemented the algorithm described in this paper in \texttt{Python}. Our code is available at 
\begin{center}
\url{https://github.com/jlouwsma/unit-fraction-expansions}.
\end{center}

\section*{Acknowledgments}

This paper was inspired by~\cite{Lev}. We would like to thank Claire Levaillant for a number of helpful conversations, about both the results of~\cite{Lev} and the results described here. The author was partially supported by a Niagara University Summer Research Award.

\newpage
\appendix
\counterwithin{figure}{section}

\section{Solutions for \texorpdfstring{$k=3$}{k=3} and small \texorpdfstring{$n$}{n}}\label{app:tree3}

\begin{amssidewaysfigure}
\centering
\begin{tikzpicture}[scale=2.16712]
\node at (-1.075,4.9) {$n=1$};
\node at (-1.075,3.4) {$n=2$};
\node at (-1.075,1.8) {$n=3$};
\node at (-1.075,0) {$n=4$};
\node (a) at (3,4.9) {\textcolor{red}{$\renewcommand{\arraystretch}{1.3}
\begin{NiceArray}{ccc}[hvlines,corners]
1 & 0 & 0
\end{NiceArray}$}};
\node (b) at (3,3.4) {\textcolor{red}{$\renewcommand{\arraystretch}{1.3}
\begin{NiceArray}{ccc}[hvlines,corners]
0 & 2 & 0
\end{NiceArray}$}};
\node (c) at (0.5,1.8) {$\renewcommand{\arraystretch}{1.3}
\begin{NiceArray}{ccc}[hvlines,corners]
3 & 0 & 0\\
0 & 0 & 0
\end{NiceArray}$};
\node (d) at (3,1.8) {\textcolor{red}{$\renewcommand{\arraystretch}{1.3}
\begin{NiceArray}{ccc}[hvlines,corners]
0 & 1 & 2
\end{NiceArray}$}};
\node (e) at (5.55,1.8) {$\renewcommand{\arraystretch}{1.3}
\begin{NiceArray}{ccc}[hvlines,corners]
1 & 1 & 0\\
0 & 1 & 0
\end{NiceArray}$};
\node (f) at (-0.075,0) {$\renewcommand{\arraystretch}{1.3}
\begin{NiceArray}{ccc}[hvlines,corners]
2 & 2 & 0\\
0 & 0 & 0
\end{NiceArray}$};
\node (g) at (1.975,0) {$\renewcommand{\arraystretch}{1.3}
\begin{NiceArray}{ccc}[hvlines,corners]
2 & 0 & 1\\
0 & 0 & 1
\end{NiceArray}$};
\node (h) at (0.95,0) {$\renewcommand{\arraystretch}{1.3}
\begin{NiceArray}{ccc}[hvlines,corners]
0 & 1 & 1\\
0 & 1 & 1
\end{NiceArray}$};
\node (i) at (3,0) {\textcolor{red}{$\renewcommand{\arraystretch}{1.3}
\begin{NiceArray}{ccc}[hvlines,corners]
0 & 0 & 4
\end{NiceArray}$}};
\node (j) at (4.025,0) {$\renewcommand{\arraystretch}{1.3}
\begin{NiceArray}{ccc}[hvlines,corners]
1 & 1 & 0\\
0 & 0 & 2
\end{NiceArray}$};
\node (k) at (5.05,0) {$\renewcommand{\arraystretch}{1.3}
\begin{NiceArray}{ccc}[hvlines,corners]
1 & 0 & 2\\
0 & 1 & 0
\end{NiceArray}$};
\node (l) at (6.075,0) {$\renewcommand{\arraystretch}{1.3}
\begin{NiceArray}{ccc}[hvlines,corners]
1 & 1 & 0\\
1 & 0 & 0\\
0 & 1 & 0
\end{NiceArray}$};
\node (m) at (7.1,0) {$\renewcommand{\arraystretch}{1.3}
\begin{NiceArray}{ccc}[hvlines,corners]
0 & 3 & 0\\
0 & 1 & 0
\end{NiceArray}$};
\draw (a) edge node[left] {1} (b);
\draw (a) edge node[above left] {2} (c);
\draw (b) edge node[left] {1} (d);
\draw (b) edge node[above right] {4} (e);
\draw (c) edge node[left=2pt] {1} (f);
\draw (d) edge node[below right] {3} (g);
\draw (d) edge node[above left] {5} (h);
\draw (d) edge node[right] {1} (i);
\draw (d) edge node[right=2pt] {4} (j);
\draw (e) edge node[left=2pt] {1} (k);
\draw (e) edge node[below left] {4} (l);
\draw (e) edge node[above right] {1} (m);
\end{tikzpicture}
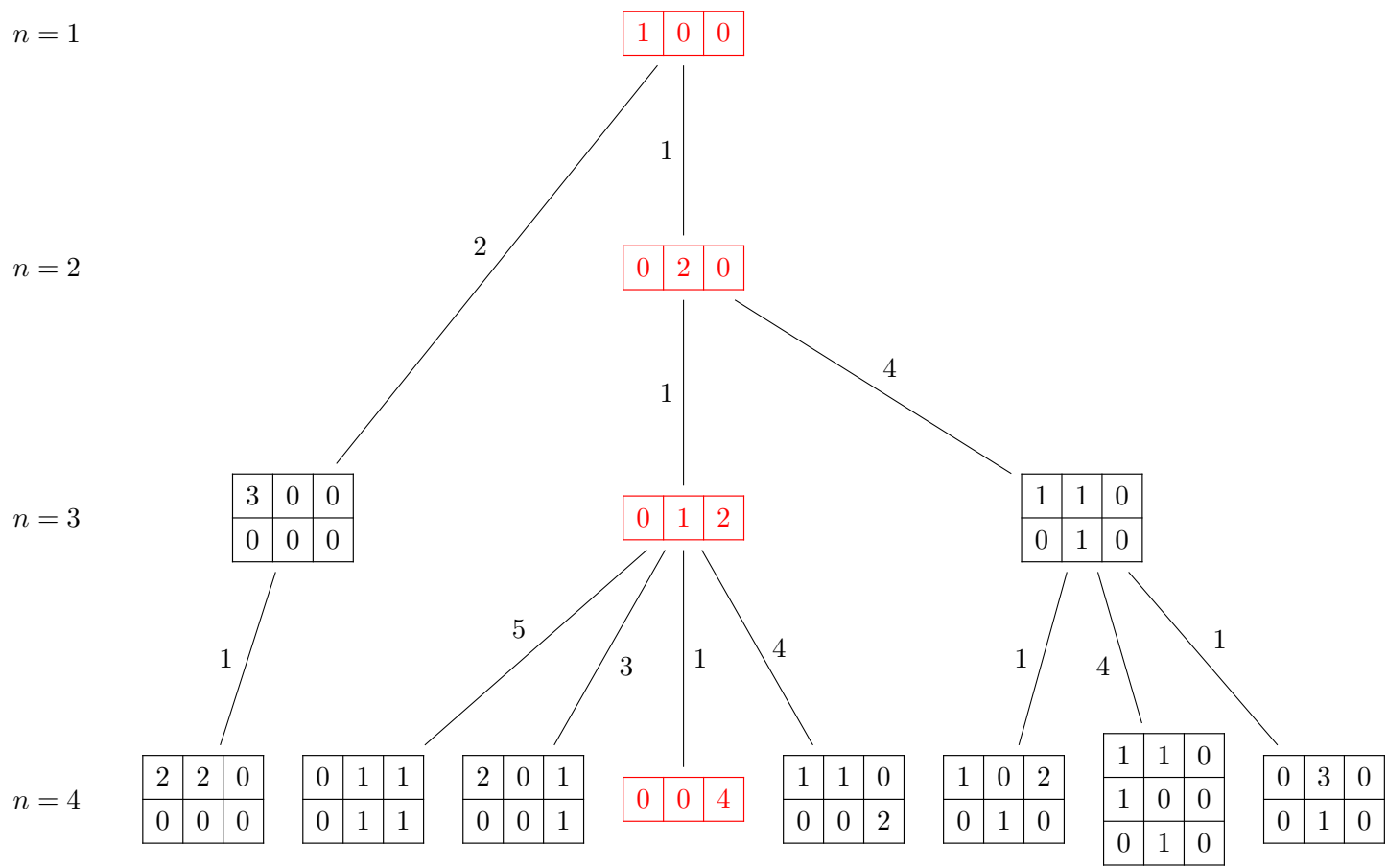
\caption{Arrays corresponding to solutions of~\eqref{eq:main} with $n\leq4$ and each $x_i\in S(3)$. The labels on the edges are the numbers of the moves used. The trivial solutions are shown in red.\label{fig:tree3}}
\end{amssidewaysfigure}

\section{Solutions for \texorpdfstring{$k=5$}{k=5} and small \texorpdfstring{$n$}{n}}\label{app:tree5}

\begin{amssidewaysfigure}
\begin{tikzpicture}[scale=2.16712]
\node at (-1,4.9) {$n=1$};
\node at (-1,3.775) {$n=2$};
\node at (-1,2.65) {$n=3$};
\node at (-1,1.425) {$n=4$};
\node at (-1,0) {$n=5$};
\node (a) at (3.6,4.9) {\textcolor{red}{$\renewcommand{\arraystretch}{1.3}
\begin{NiceArray}{ccc}[hvlines,corners]
1 & 0 & 0
\end{NiceArray}$}};
\node (b) at (3.6,3.775) {\textcolor{red}{$\renewcommand{\arraystretch}{1.3}
\begin{NiceArray}{ccc}[hvlines,corners]
0 & 2 & 0
\end{NiceArray}$}};
\node (c) at (3.6,2.65) {\textcolor{red}{$\renewcommand{\arraystretch}{1.3}
\begin{NiceArray}{ccc}[hvlines,corners]
0 & 1 & 2
\end{NiceArray}$}};
\node (d) at (2.1,1.425) {$\renewcommand{\arraystretch}{1.3}
\begin{NiceArray}{ccc}[hvlines,corners]
1 & 0 & 1\\
0 & 1 & 1
\end{NiceArray}$};
\node (e) at (3.6,1.425) {\textcolor{red}{$\renewcommand{\arraystretch}{1.3}
\begin{NiceArray}{ccc}[hvlines,corners]
0 & 0 & 4
\end{NiceArray}$}};
\node (f) at (6.3,1.425) {$\renewcommand{\arraystretch}{1.3}
\begin{NiceArray}{ccc}[hvlines,corners]
2 & 1 & 0\\
0 & 1 & 0
\end{NiceArray}$};
\node (g) at (0,0) {$\renewcommand{\arraystretch}{1.3}
\begin{NiceArray}{ccc}[hvlines,corners]
5 & 0 & 0\\
0 & 0 & 0
\end{NiceArray}$};
\node (h) at (1.2,0) {$\renewcommand{\arraystretch}{1.3}
\begin{NiceArray}{ccc}[hvlines,corners]
0 & 2 & 1\\
0 & 1 & 1
\end{NiceArray}$};
\node (i) at (2.4,0) {$\renewcommand{\arraystretch}{1.3}
\begin{NiceArray}{ccc}[hvlines,corners]
1 & 0 & 1\\
1 & 0 & 0\\
0 & 1 & 1
\end{NiceArray}$};
\node (j) at (3.6,0) {$\renewcommand{\arraystretch}{1.3}
\begin{NiceArray}{ccc}[hvlines,corners]
1 & 0 & 1\\
0 & 0 & 3
\end{NiceArray}$};
\node (k) at (4.8,0) {$\renewcommand{\arraystretch}{1.3}
\begin{NiceArray}{ccc}[hvlines,corners]
2 & 1 & 0\\
0 & 0 & 2
\end{NiceArray}$};
\node (l) at (6,0) {$\renewcommand{\arraystretch}{1.3}
\begin{NiceArray}{ccc}[hvlines,corners]
1 & 3 & 0\\
0 & 1 & 0
\end{NiceArray}$};
\node (m) at (7.2,0) {$\renewcommand{\arraystretch}{1.3}
\begin{NiceArray}{ccc}[hvlines,corners]
2 & 0 & 2\\
0 & 1 & 0
\end{NiceArray}$};
\draw (a) edge node[left] {1} (b);
\draw (b) edge node[left] {1} (c);
\draw (c) edge node[above left] {3} (d);
\draw (c) edge node[left] {1} (e);
\draw (b) edge node[above right] {4} (f);
\draw (a) edge node[above left] {2} (g);
\draw (d) edge node[above left] {1} (h);
\draw (d) edge node[right=2pt] {3} (i);
\draw (e) edge node[left] {3} (j);
\draw (c) edge node[right=3pt] {4} (k);
\draw (f) edge node[left=2pt] {1} (l);
\draw (f) edge node[above right] {1} (m);
\end{tikzpicture}
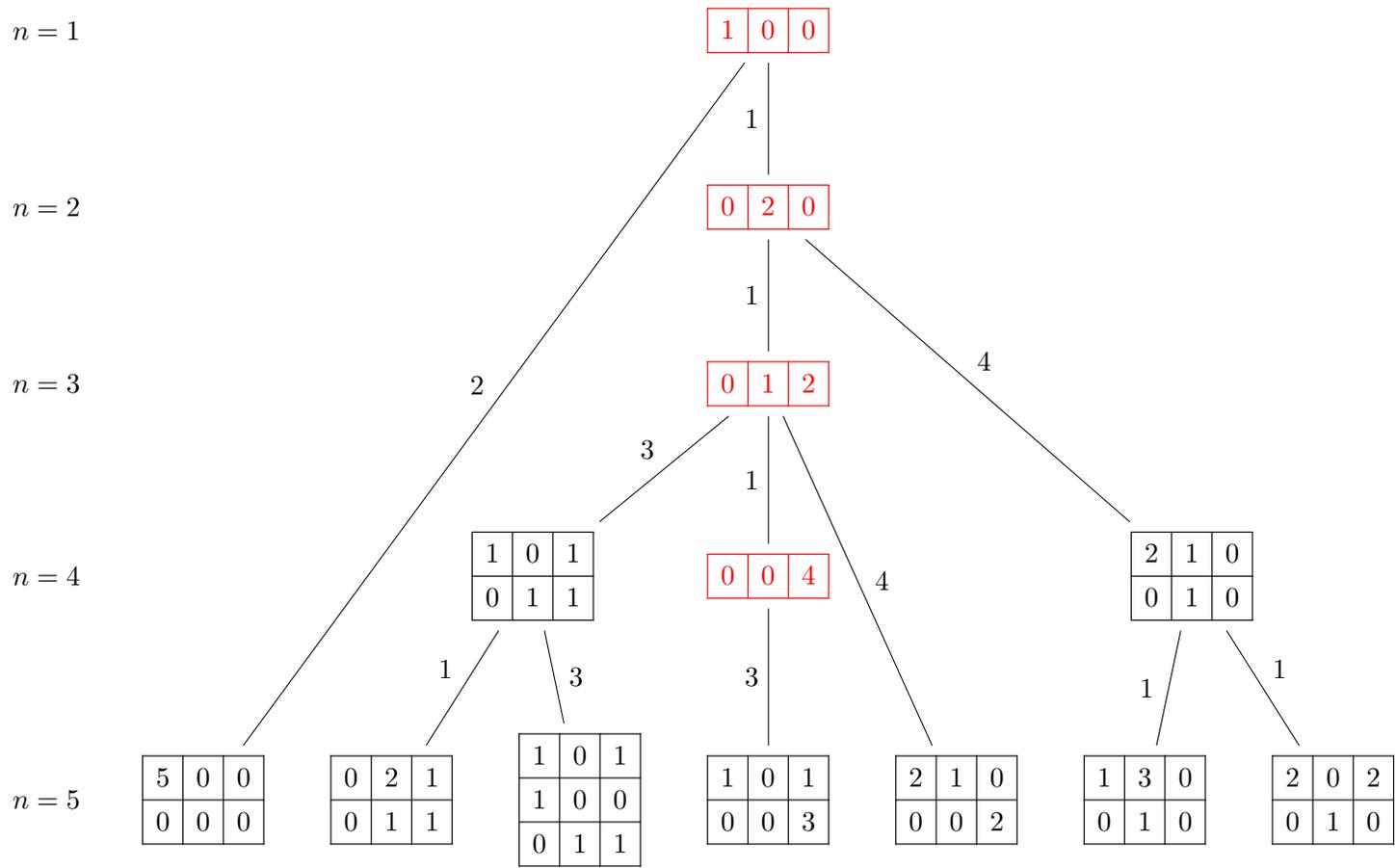
\caption{Arrays corresponding to solutions of~\eqref{eq:main} with $n\leq5$ and each $x_i\in S(5)$. The labels on the edges are the numbers of the moves used. The trivial solutions are shown in red.\label{fig:tree5}}
\end{amssidewaysfigure}

\bibliographystyle{amsplain}
\bibliography{UnitFractionExpansions}

\end{document}